\documentclass{article}
\usepackage{arxiv}

\usepackage{amsmath,amsfonts,amssymb,amsthm}
\usepackage{lipsum}
\usepackage{amsfonts}
\usepackage{graphicx}
\usepackage{amsmath}
\usepackage{amssymb}
\usepackage{epstopdf}
\usepackage{mathrsfs}
\usepackage{algorithmic}
\ifpdf
  \DeclareGraphicsExtensions{.eps,.pdf,.png,.jpg}
\else
  \DeclareGraphicsExtensions{.eps}
\fi

\usepackage{enumitem}
\setlist[enumerate]{leftmargin=*}
\setlist[itemize]{leftmargin=*}

\RequirePackage{algorithm}[0.1]
\usepackage{graphics,graphicx}
\usepackage{epstopdf}
\usepackage{hyperref}
\usepackage{cleveref}

\usepackage{amsopn}
\DeclareMathOperator{\diag}{\mathrm{diag}}
\DeclareMathOperator{\R}{\mathbb{R}}

\newtheorem{Th}{Theorem}[section]
\newtheorem{Proposition}[Th]{Proposition}
\newtheorem{Lemma}[Th]{Lemma}

\newcommand{\algname}{Bilevel VarPro}

\title{Constrained Variable Projection for Structured Problems} 

\author{%
\textbf{Emanuele Zangrando}\\
Gran Sasso Science Institute,  L'Aquila, Italy.\\ 
    \And
    \textbf{Sara Venturini}\\
    MOBS Lab, Northeastern University, Boston, USA. \\
    \AND
    \textbf{Francesco Rinaldi}\\
    University of Padova, Padova, Italy.\\
    \And 
    \textbf{Francesco Tudisco}\\
    Gran Sasso Science Institute, L'Aquila, Italy.\\
    University of Edinburgh, Edinburgh, UK.
}

\begin{document}

\maketitle

\begin{abstract}

Variable projection is a classical technique for separable nonlinear least-squares problems, in which variables that enter linearly are eliminated exactly, yielding a reduced nonlinear problem. By expressing this framework as a particular instance of a broader class of bilevel optimization problems, we develop a constrained variable-projection framework for data-science models, where the remaining variables are subject to convex constraints and the eliminated variables arise from a lower-level least-squares problem. In particular, by interpreting variable projection as a collapsed bilevel optimization problem, we derive exact reduced-gradient formulas compatible with automatic differentiation and propose a conditional-gradient algorithm for the resulting constrained reduced problem. We establish convergence guarantees under standard smoothness and compactness assumptions, and discuss extensions to structured lower-level variables. Numerical experiments on sparse autoencoding, dictionary learning, blind deconvolution, and few-shot learning suggest that the method can improve wall-clock efficiency and data efficiency relative to natural joint-optimization baselines.
\end{abstract}

\section{Introduction}\label{sec:introduction}
Many problems in modern data science involve fitting models with two distinct types of variables: a variable block that enters a loss function linearly or quadratically, and a second block of nonlinear variables that controls features, representations, constraints, or physical parameters. This structure appears in applications such as dictionary learning \cite{Rubinstein_DL}, inverse problems \cite{Chavent2009-aa}, signal recovery, representation learning \cite{snlls21}, and neural network training \cite{ruthotto_trainvarpro}. In these settings, the linear variable can often be optimized exactly once the nonlinear variables are fixed, yielding a reduced optimization problem of smaller dimension.

A classical framework for exploiting this structure is the variable projection method, or VarPro, initially introduced for separable nonlinear least-squares problems \cite{Gobub_Pereyra_varpro,Gobub_Veque_varpro,golub2003separable}. In its standard form, VarPro applies to problems of the type
\begin{equation}\label{eq:joint_formulation}
    \min_{w \in \R^{N},\,\theta \in \R^{p}} \frac12 \|M(\theta)w-y(\theta)\|^2 + \frac{\lambda}{2}\|\Omega w\|^2 + \frac{\mu}{2}R(\theta),
\end{equation}
where $w$ is a linear variable and $\theta$ is a nonlinear variable. For fixed $\theta$, the minimizer with respect to $w$ can be computed by solving a least-squares problem. Substituting this minimizer into the objective produces a reduced problem in $\theta$ alone. This reduction can substantially decrease the dimension of the optimization problem and can improve numerical conditioning and computational efficiency.

The classical VarPro theory, however, is primarily designed for \emph{unconstrained} separable nonlinear least-squares problems. This leaves open an important question: how should variable projection be used when the remaining nonlinear variables are constrained? Such constraints are common and often essential, as structured feasible sets arise naturally in signal processing, inverse problems, and machine learning. For instance, sparsity, rank, and norm constraints promote efficiency, interpretability, and regularization \cite{candes2011robust,donoho2006compressed,hastie2015statistical,recht2010guaranteed,tibshirani1996regression,zangrando2025debora}, while simplex, box, and non-negativity constraints encode feasibility or statistical structure \cite{carlini2017towards,duchi2008efficient,gillis2020nonnegative,jaggi2014equivalence,lim2016box}.

The goal of this paper is to develop a constrained variable-projection framework for such problems. We reinterpret and generalize VarPro through a collapsed bilevel optimization problem: the lower-level problem eliminates the variable $w$ through a structured least-squares solve, while the upper-level problem optimizes the remaining variable $\theta$ over a convex feasible set. This viewpoint allows us to combine
the dimension-reduction benefits of VarPro with projection-free constrained optimization methods. In particular, we derive reduced-gradient formulas that can be evaluated efficiently using a combination of closed-form lower-level solvers and automatic differentiation through vector-Jacobian products. We then use these gradients inside a conditional gradient, or Frank-Wolfe, method for the reduced constrained problem.

The resulting framework is motivated by applications in which the eliminated variable has a clear statistical or computational meaning. In sparse autoencoding, the eliminated variable may correspond to a linear decoder or readout map. In dictionary learning, the eliminated variables encode representation coefficients or dictionary-dependent least-squares updates. In blind deconvolution and inverse problems, they correspond to structured linear operators or signal components. In few-shot learning, a pretrained nonlinear representation can be combined with an optimized linear head, leading naturally to a variable-projection formulation. 

\subsection{Contributions}
The main contributions of this paper are as follows.
\begin{itemize}
    \item We observe that the problem treated in the original Variable-Projection works is the collapsed version of a specific class of bilevel optimization problems. Following this insight, we formulate constrained variable projection as a collapsed bilevel optimization problem, in which the lower-level problem is a structured least-squares problem and the upper-level problem imposes convex constraints on the remaining variables.

    \item We derive explicit reduced-gradient formulas for the constrained reduced objective. The formulas combine closed-form lower-level solutions with automatic differentiation through vector-Jacobian products, avoiding, when possible, the need to differentiate naively through ill-conditioned normal equations.

    \item We propose a tailored conditional-gradient method for the constrained reduced problem. The method is projection-free and is therefore well-suited to feasible sets for which linear minimization oracles are cheaper than Euclidean projections. We further establish convergence guarantees for the proposed method under standard assumptions.

    \item We demonstrate the flexibility of the framework on representative data-science and machine-learning problems, including sparse autoencoding, dictionary learning, blind deconvolution, and few-shot learning with a
    pretrained neural representation.
\end{itemize}

\subsection{Organization}
The rest of the paper is organized as follows. In
\Cref{sec:related-work}, we discuss related work on variable projection, bilevel optimization, implicit differentiation, and conditional-gradient methods. In \Cref{sec:problem_formulation}, we introduce the general constrained variable-projection problem formulation and its bilevel interpretation. We then derive
the reduced-gradient expressions used by our method and discuss extensions to structured lower-level variables. The convergence analysis of the proposed conditional-gradient scheme is presented next. Finally, we report numerical experiments on several representative data-science problems.

\section{Related Work}\label{sec:related-work}
The variable projection method was developed for separable nonlinear least-squares problems in which some variables enter the model linearly and can therefore be eliminated exactly. The observation that such models can be reduced to a nonlinear problem in the remaining variables appears in \cite{lawton_varpro}, which credits an unpublished 1965 work by N.\,E.~Dahl for the original idea. Golub and Pereyra developed the method systematically in \cite{Gobub_Pereyra_varpro}, deriving formulas for the Jacobian of the reduced residual and proposing practical Gauss-Newton-type algorithms. The method and
its developments are surveyed in \cite{golub2003separable}.

A key computational issue in VarPro is the cost of forming the reduced Jacobian. Kaufman introduced a cheaper approximation for the Jacobian under a small-residual assumption \cite{Kaufman1975}. More recent work has revisited this approximation and its effect on convergence. In particular, \cite{CHEN2025112300} studies the large-residual regime and proposes improved variants with stronger convergence
properties, \cite{espanol2025local} analyzes the effect of approximate Jacobians and proposes practical stopping criteria that preserve convergence, and \cite{Leewen_varprononsmooth} extends the analysis for nonsmooth objectives. Global optimality results for a Riemannian relaxed version of the problem have been proposed in \cite{dus2026grassmanniangeometryglobalconvergence}, while an algorithm for sparse regularization has been proposed in \cite{Xu2024}.

VarPro has been used successfully in a range of applications, including atmospheric remote sensing \cite{barligea_varpro_remotesensing}, neural network training \cite{DONG2022115284,ruthotto_trainvarpro,ruthotto_slimtrain},
semi-blind image deblurring \cite{salzer2026variableprojectionmethodssolving}, low-rank matrix approximation problems \cite{USEVICH2014430,Markovsky_advances}, and polynomial nearness problems \cite{USEVICH2017176}.
These works show that eliminating linear variables can improve the numerical behavior of large-scale learning and inverse problems. 

Our work follows this general direction, with the emphasis on constrained data-science models in which the remaining variables typically satisfy convex structural constraints. 
Sparse autoencoding highlights the interaction between representation learning and regularized linear reconstruction. Dictionary learning illustrates the role of constraints in structured matrix factorization. Blind deconvolution connects the
framework to inverse problems and signal processing. Few-shot learning shows how a nonlinear representation map can be combined with an optimized linear head.

\paragraph{Bilevel optimization and implicit differentiation}
The formulation considered in this paper can be interpreted as a bilevel optimization problem in which the lower-level problem is solved exactly and then substituted into the upper-level objective (see, e.g., \cite{bard1998practical,colson2007overview,dempe2020bilevel}). This places constrained VarPro within the broader class of collapsed bilevel methods, argmin differentiation, and differentiable optimization layers (see, e.g., \cite{agrawal2019differentiable,amos2017optnet,blondel2022efficient,franceschi2018bilevel,gould2016differentiating} for further details on these topics). In contrast to generic bilevel optimization, the lower-level problems considered here have a least-squares structure, which allows explicit characterization of the lower-level solution and efficient computation of reduced gradients. 
This structure is central to the algorithmic efficiency and theoretical developments of the paper.

\paragraph{Conditional-gradient methods for constrained data science}
The upper-level feasible sets that arise in data science are often convex but not necessarily easy to project onto. Conditional-gradient methods are attractive in this setting because they replace Euclidean projections with linear minimization oracles \cite{bomze2021frank, braun2025conditional, Frank1956-fg}. Such methods are particularly useful for constraints involving simplices, norm balls, linear structures, and other sets for which linear optimization is cheaper than projection (see, e.g., \cite{combettes2021complexity,jaggi2013revisiting}). Our algorithm applies a tailored conditional-gradient method to the reduced VarPro objective, thereby combining projection-free constrained optimization with exact elimination of the least-squares variable.

\section{Problem setting}\label{sec:problem_formulation}

\subsection{Notation}

Throughout the paper, $\|\cdot\|$ denotes the $\ell^2$ componentwise norm, unless stated otherwise. For a full
column-rank matrix $Z$, we write $Z^+$ for its Moore-Penrose pseudoinverse and $\mathcal P(Z) := Z Z^+$ for the orthogonal projector onto the range of $Z$.

\subsection{General bilevel problem formulation}
We consider constrained separable nonlinear least-squares problems of the form
\begin{equation}\label{eq:varpro_general_formulation}
    \min_{w \in \R^{N},\,\theta \in \mathcal C}
    \frac12 \|M(\theta)w-y(\theta)\|^2
    + \frac{\lambda}{2}\|\Omega w\|^2
    + \frac{\mu}{2}R(\theta),
\end{equation}
where $\mathcal C \subseteq \R^p$ is a convex feasible set,
$M:\R^p \to \R^{d\times N}$, $y:\R^p \to \R^d$,
$\Omega\in \mathrm{GL}(\R^N)$, and $\lambda,\mu\geq 0$ are regularization
parameters. The variable $w$ enters the data-fitting term linearly, whereas
$\theta$ enters nonlinearly through $M(\theta)$ and $y(\theta)$.

The Tikhonov term in $w$ can be absorbed into the least-squares residual. Indeed,
\eqref{eq:varpro_general_formulation} is equivalent to
\begin{equation}\label{eq:augmented_varpro_formulation}
    \min_{w \in \R^{N},\,\theta \in \mathcal C}
    \frac12 \|M_\lambda(\theta)w-y_\lambda(\theta)\|^2
    + \frac{\mu}{2}R(\theta),
\end{equation}
where $ M_\lambda : \R^p \to \R^{(N+d) \times N}, M_\lambda(\theta) = [M(\theta)^\top,\sqrt\lambda \Omega^\top]^\top,  y_\lambda(\theta) = [y(\theta)^\top,0_N^\top]^\top$. 
When $\lambda>0$, the matrix $M_\lambda(\theta)$ has full column rank for every
$\theta$, as $M_\lambda(\theta)^\top M_\lambda(\theta)
    \succeq
    \lambda\,\Omega^\top\Omega$, 
and therefore its smallest eigenvalue is bounded below by
$\lambda$ times the smallest eigenvalue of $\Omega^\top\Omega$, which must be positive. 

For this reason, and to simplify notation, we henceforth work directly with the augmented formulation. That is, unless stated otherwise, $M$ and $y$ denote the
augmented quantities $M_\lambda$ and $y_\lambda$, and we assume that $M:\R^p \to \R^{m\times N}$ and $y:\R^p \to \R^m$, with $m=d+N$ and $M(\theta)$ full column rank for every $\theta\in\mathcal C$. With this convention, the problem becomes
\begin{equation}\label{eq:varpro_augmented_convention}
    \min_{w \in \R^N,\,\theta\in\mathcal C}
    \frac12\|M(\theta)w-y(\theta)\|^2
    + \frac{\mu}{2}R(\theta).
\end{equation}

Problem \eqref{eq:varpro_augmented_convention} admits an equivalent bilevel interpretation. For a fixed value of $\theta$, the optimal linear variable is
obtained by solving the lower-level least-squares problem
\[
    \widehat w(\theta)
    \in
    \underset{w \in \R^N}{\arg\min}
    \frac12\|M(\theta)w-y(\theta)\|^2.
\]
Thus \eqref{eq:varpro_augmented_convention} can be written as
\begin{equation}\label{eq:bilevel_varpro_formulation}
\begin{cases}
    \displaystyle
    \min_{\theta \in \mathcal C}
    \frac12 \|M(\theta)\widehat w-y(\theta)\|^2
    + \frac{\mu}{2}R(\theta),
    \\[0.8em]
    \displaystyle
    \widehat w
    \in
    \underset{w \in \R^N}{\arg\min}
    \frac12 \|M(\theta)w-y(\theta)\|^2.
\end{cases}
\end{equation}
Since $M(\theta)$ has full column rank, the lower-level solution is unique and is
given by
\begin{equation}\label{eq:closed_form_lower_solution}
    \widehat w(\theta)
    =
    \bigl(M(\theta)^\top M(\theta)\bigr)^{-1}M(\theta)^\top y(\theta)
    =
    M(\theta)^+y(\theta).
\end{equation}
Substituting \eqref{eq:closed_form_lower_solution} into the upper-level
objective collapses the bilevel problem to the reduced single-level problem
\begin{equation}\label{eq:unrolled}
    \min_{\theta\in\mathcal C}
    f(\theta)
    :=
    \frac12
    \bigl\|
        \bigl(\mathcal P(M(\theta))-I\bigr)y(\theta)
    \bigr\|^2
    +
    \frac{\mu}{2}R(\theta),
\end{equation}
where $\mathcal P(M(\theta))=M(\theta)M(\theta)^+$ is the orthogonal projector
onto $\operatorname{range}(M(\theta))$. 
This is the constrained analogue of the classical variable-projection formulation of
\cite{Gobub_Pereyra_varpro}.

The bilevel viewpoint also suggests a broader class of problems. In particular, we will consider a generalized formulation in which the lower-level problem is used to define
$\widehat w(\theta)$ and the upper-level objective need not involve the same
least-squares operator:
\begin{equation}\label{eq:bilevel_general_formulation}
\begin{cases}
    \displaystyle
    \min_{\theta \in \mathcal C}
    f_U(\widehat w,\theta)
    :=
    \frac12 \|M_U(\theta)\widehat w-y_U(\theta)\|^2
    + \frac{\mu}{2}R(\theta),
    \\[0.8em]
    \displaystyle
    \widehat w
    \in
    \underset{w \in \R^N}{\arg\min}
    f_L(w,\theta)
    :=
    \frac12 \|M_L(\theta)w-y_L(\theta)\|^2.
\end{cases}
\end{equation}
Here $M_L,M_U:\R^p\to\R^{m\times N}$ are assumed to have full column rank for
all relevant $\theta$, and $y_L,y_U:\R^p\to\R^m$. The classical constrained
VarPro formulation corresponds to the special case
$M_L=M_U=M$ and $y_L=y_U=y$.

In the remainder of the paper, we focus on the case in which the feasible set $\mathcal C\subseteq\R^p$
factorizes as
\[
    \mathcal C = \mathcal C_A \times \mathcal C_B
    \subseteq \R^{p_A}\times\R^{p_B},
    \qquad
    p_A+p_B=p,
\]
where $\mathcal C_A\subset\R^{p_A}$ is compact and convex, and
$\mathcal C_B=\R^{p_B}$. This structure separates the variables subject to
explicit constraints from those that remain unconstrained, and it will be used
in the design and analysis of the conditional-gradient method below.

\subsection{Constrained VarPro and conditional gradient methods}

Under the assumption that the feasible set $\mathcal C$ is convex, the reduced problem can be naturally approached using conditional-gradient methods. In particular, Frank-Wolfe-type methods \cite{bomze2021frank,Frank1956-fg} require only the solution of a linear minimization oracle (LMO) at each iteration; that is, they minimize a first-order approximation of the objective over the original feasible set. This makes them especially attractive when projections onto $\mathcal C$ are expensive, but linear minimization over $\mathcal C$ is efficient.

The main algorithmic requirement is therefore the ability to evaluate the gradient of the reduced upper-level objective, $\nabla \widehat f_U(\theta) = \nabla f_U(\widehat w(\theta),\theta)$. Assuming that $\nabla R$ can be computed efficiently, the central difficulty is the evaluation of the contribution coming from the dependence of $\widehat w(\theta)$ on $\theta$. In principle, this derivative could be computed by automatic differentiation through the closed-form expression for $\widehat w(\theta)$. However, this approach may be numerically unstable when $M^\top M$ has small singular values.

To retain flexibility across applications while avoiding this instability, we use partial automatic differentiation. The starting point is the compositional structure of the reduced objective\footnote{For notational simplicity, in the calculations we only display the least-squares part of the function.}:
\begin{align*}
\widehat f_U(\theta) =& \|M_U(\theta)M_L(\theta)^+ y_L(\theta) - y_U(\theta) \|^2 = \| (M_U M_L^+ - \mathcal P(M_U)) y_L \|^2 + \|\mathcal P(M_U)(y_L-y_U) \|^2 +\\ +&\|(\mathcal P(M_U)-I)y_U \|^2 + 2 \langle (M_U M_L^+ - \mathcal P(M_U)) y_L, {\mathcal P(M_U)(y_L-y_U)} \rangle,
\end{align*}
whose gradient can be written by the chain rule as
\begin{align}\label{eq:gradient_chain_rule}
&\nabla \widehat f_U(\theta) = \nabla_\theta f_U|_{(\widehat w(\theta),\theta)}+\partial_\theta\widehat w(\theta)^\top \nabla_w f_U|_{(\widehat w(\theta),\theta)} .
\end{align}
The derivative $\partial_\theta \widehat w(\theta)$ can be obtained by differentiating the first-order stationarity conditions for the lower-level problem. Namely,
\begin{align*}
    \partial_\theta \nabla_w f_L(\widehat w(\theta),\theta) = \nabla^2_{w\theta} f_L|_{(\widehat w(\theta),\theta)}+ \nabla^2_{ww} f_L|_{(\widehat w(\theta),\theta)} \partial_\theta \widehat w(\theta),
\end{align*}
which implies
$\partial_\theta \widehat w(\theta) = -(\nabla^2_{ww} f_L|_{(\widehat w(\theta),\theta)})^{-1}\nabla^2_{w\theta} f_L|_{(\widehat w(\theta),\theta)}$.
Substituting this expression into the chain rule \eqref{eq:gradient_chain_rule}, and setting $\Gamma:= M_U(M_L^\top M_L)^{-1}$,  gives
\begin{align}\label{eq:gradient_structure}
\nabla \widehat f_U(\theta) =&\nabla_\theta f_U|_{(\widehat w(\theta),\theta)}- \nabla^2_{w\theta} f_L|_{(\widehat w(\theta),\theta)}^\top(\nabla^2_{ww} f_L|_{(\widehat w(\theta),\theta)})^{-\top} \nabla_w f_U|_{(\widehat w(\theta),\theta)}  \\ \nonumber =&
\nabla_\theta f_U|_{(\widehat w(\theta),\theta)}- \nabla^2_{w\theta} f_L|_{(\widehat w(\theta),\theta)}^\top (M_L^\top M_L)^{-\top} \bigl[ M_U^\top M_U (M_L^\top M_L)^{-1}M_L^\top y_L-M_U^\top y_U  \bigr]  \\  \nonumber =& \nabla_\theta f_U|_{(\widehat w(\theta),\theta)}- \nabla^2_{\theta w} f_L|_{(\widehat w(\theta),\theta)} \bigl[ \Gamma^\top \Gamma M_L^\top y_L-\Gamma^\top y_U   \bigr]  \\  \nonumber =& \nabla_\theta f_U|_{(\widehat w(\theta),\theta)}- \nabla^2_{\theta w} f_L|_{(\widehat w(\theta),\theta)} \Gamma^\top\bigl[  M_U M_L^+y_L-y_U
   \bigr] \, .
\end{align}  

In the special case where $M_L=M_U$ and $y_L=y_U$, one has
$\Gamma^\top\bigl[  M_U M_L^+y_L-y_U
   \bigr]=0$, and the corresponding coupling term in
\Cref{eq:gradient_structure} vanishes, leading to
$\nabla \widehat f_U(\theta) =
\nabla_\theta f_U|_{(\widehat w(\theta),\theta)}$.

The form of \Cref{eq:gradient_structure} is useful computationally. In any application where $\widehat w(\theta)$ can be evaluated efficiently and where one can compute the action of
$\nabla^2_{\theta w} f_L|_{(\widehat w(\theta),\theta)}$ on a tangent vector, the reduced gradient $\nabla \widehat f_U(\theta)$ can be evaluated exactly using automatic differentiation. As we will discuss in \Cref{sec:experiments}, this covers several applications of interest. In particular, to compute the required action of the Hessian on a vector, it is sufficient to notice that
\[
\nabla^2_{w \theta} f_L^\top[v] = \partial_\theta \langle \nabla_w f_L,v \rangle = \partial_{\theta} \langle M_L(\theta)^\top (M_L(\theta) w -y_L(\theta)),v \rangle .
\]
 Therefore, if we have access to an automatic differentiation system, the vector-Jacobian products can be queried easily through the previous formula. 

\subsection{Extension for structured matrix problems} \label{sec:structured_varpro}

In several applications, the eliminated variable $w$ is not free in the ambient space, but is constrained to belong to a lower-dimensional set, which we denote by $\mathscr W$. This set may have the structure of a differentiable manifold, an affine space, or a linear subspace. The preceding framework continues to apply whenever the lower-level problem has a unique solution. In particular, the same analysis can be used for any geometric class $\mathscr W$ such that
\[
    \min_{w \in \mathscr W} \|M w-y \|_2^2 + \lambda \|\Omega w \|_2^2
\]
admits a unique minimizer. Up to an additive constant, this lower-level problem can be equivalently written as
\[
    \min_{w \in \mathscr W} \|w-\widehat  w \|^2_Q,
\]
where $Q = ( M^\top M + \lambda \Omega^\top \Omega ), \| x\|_Q^2 = \| Q^{1/2} x \|_2^2, \widehat w = Q^{-1}M^\top y$.
Thus, the lower-level solution is the projection of $\widehat w$ onto $\mathscr W$ with respect to the $Q$-inner product. When $\lambda>0$, the matrix $Q$ is positive definite and $\|\cdot\|_Q$ is a norm; otherwise, it may only define a seminorm. In particular, if $\mathscr W$ is an affine subspace, the corresponding projection depends smoothly on $M$.

The situation is more delicate when $\mathscr W$ is only convex. In finite dimension, closed convex sets are Chebyshev, so the projection is single-valued and continuous. However, differentiability of the projection may fail. In particular, the next result characterizes when the differentiability of the lower-level solution holds globally.
\begin{Proposition}(Differentiability of optimal solution for constrained problems)\label{prop:differentiability_optimalconstrained} 
Let $\mathscr W$ be a Chebyshev subset of $E:=\R^N$ and consider the problem
    \[
    \mathcal P_{\mathscr W}^{Q}(\widehat w):= \arg\min_{w \in \mathscr W} \|w-\widehat w \|_{Q}^2.
    \]
Then, for $Q = (M^\top M + \lambda \Omega^\top  \Omega), \| x\|_Q^2 = \|Q^{1/2}x \|_2^2, \widehat w = Q^{-1} M^\top y$, and assuming that $\lambda > 0, m \geq N$ and $\Omega$ full rank, the map 
\[
\Pi : \R^{m \times N} \times \R^{m} \to \R^{N},\quad  \Pi(M,y) = \mathcal P^{Q(M)}(\widehat w(M,y))
\]
is differentiable for all $(M,y) \in \R^{m \times N} \times \R^{m}$ if and only if $\mathscr W$ is affine.
\end{Proposition}

\begin{proof}
Let \(E=\mathbb R^{N}\), endowed with the $\ell^2$ inner product.
Since \(\mathscr W\) is a Chebyshev subset of a finite-dimensional Euclidean
space, \(\mathscr W\) is closed and convex. Hence, for every positive definite
matrix \(Q\), the \(Q\)-metric projection onto \(\mathscr W\) is well-defined
and single-valued.

Because \(\Omega ^\top\Omega\succ0\) and \(\lambda>0\), we have
\[
Q(M)=M^\top M+\lambda\Omega^\top\Omega\succ0
\]
for every \(M\). Thus, \(Q(M)^{-1}\) depends smoothly on \(M\), and so does $
\widehat w(M,y)= Q(M)^{-1}M^\top y.$

We first prove the easy direction. Suppose that \(\mathscr W\) is affine, i.e., $\mathscr W=w_*+\mathscr L,$ where \(\mathscr L\subset E\) is a linear subspace. Fix a basis
\(v_1,\dots,v_r\) of \(\mathscr L\). For fixed \(Q\succ0\), the projection of
\(\widehat w\) onto \(\mathscr W\) has the form
\[
\mathcal P_{\mathscr W}^Q(\widehat w)
=
w_*+\sum_{i=1}^r \alpha_i v_i.
\]
The coefficients are determined by the normal equations
\[
\langle
\Big(\widehat w-w_*-\sum_{j=1}^m \alpha_j v_j\Big)
,Q v_i
\rangle=0,
\qquad i=1,\dots,r.
\]
Equivalently, $G(Q)\alpha=b(Q,\widehat w),$ where
\[
G(Q)_{ij} = \langle v_i,v_j \rangle_{Q},
\qquad
b(Q,\widehat w)_i=\langle\widehat w-w_*, Q v_i \rangle.
\]
Since \(Q\succ0\), the Gram matrix \(G(Q)\) is positive definite, hence
invertible. Therefore
\[
\alpha=G(Q)^{-1}b(Q,\widehat W)
\]
depends smoothly on \((Q,\widehat w)\). Since \(Q(M)\) and
\(\widehat w(M,y)\) depend smoothly on \((M,y)\), the map
\[
\Pi(M,y)=\mathcal P_{\mathscr W}^{Q(M)}(\widehat w(M,y))
\]
is differentiable, indeed smooth.

Conversely, assume that \(\Pi\) is differentiable for every \((M,y)\).
Choose \(M_0\in\mathbb R^{m\times N}\) with \(\operatorname{rank}(M_0)=N\),
which is possible because \(N \leq m\). Set $Q_0:=Q(M_0)\succ0$ and consider the linear map
\[
L:\mathbb R^{m}\to \mathbb R^{N},
\qquad
L(y)= Q_0^{-1} M_0^\top y.
\]
Since \(\operatorname{rank}(M_0)=N\), the matrix \(Q_0^{-1} M_0^\top\) has rank
\(N\), hence \(L\) is surjective. Thus \(L\) admits a linear right inverse
\(R:\R^N\to\mathbb R^{m}\).
For fixed \(M_0\), the map $
y\mapsto \Pi(M_0,y) = \mathcal P_{\mathscr W}^{Q_0}(L(y))$
is differentiable by assumption. Since \(L\circ R=\operatorname{Id}_E\), we get $\mathcal P_{\mathscr W}^{Q_0}(z) =
\Pi(M_0,Rz)$.
Therefore, the full \(Q_0\)-metric projection $
z\mapsto \mathcal P_{\mathscr W}^{Q_0}(z)$ is differentiable everywhere on \(E\).

Now reduce to the ordinary Frobenius projection. Define the invertible linear
map
\[
T:E\to E,\qquad T(w)=Q_0^{1/2}w,
\]
and set $C:=T(\mathscr W)=Q_0^{1/2}\mathscr W .$
Since 
$\|w-\widehat w\|_{Q_0} =
\|T(w)-T(\widehat w)\|_2,$
we have $
T\bigl(\mathcal P_{\mathscr W}^{Q_0}(\widehat W)\bigr) =
\mathcal P_C(T\widehat W),$
or equivalently
\[
\mathcal P_C
=
T\circ \mathcal P_{\mathscr W}^{Q_0}\circ T^{-1}.
\]
Hence the ordinary Frobenius projection \(\mathcal P_C\) is differentiable
everywhere.

We now show that this forces \(C\) to be affine. Let \(c\in C\), and let $
\mathcal K_cC:=\overline{\operatorname{cone}(C-c)}
$
be the tangent cone of \(C\) at \(c\). For closed convex sets in a finite
dimensional Hilbert space, the directional derivative of the metric projection
at a point \(c\in C\) is given by $d\mathcal P_C(c;h)=\mathcal P_{\mathcal K_cC}(h),$
where \(\mathcal P_{\mathcal K_cC}\) denotes the orthogonal projection onto the closed
convex cone \(\mathcal K_cC\).

But \(\mathcal P_C\) is Frechet differentiable at \(c\), so
$h\mapsto d\mathcal P_C(c;h)$
is linear, i.e., \(h\mapsto \mathcal P_{\mathcal K_cC}(h)\) is linear. The range of
this linear map is exactly \(\mathcal K_cC\), which is therefore a linear subspace.
Let now $A:=\operatorname{aff}(C)$
be the affine hull of \(C\), and let $
V:=A-c=\operatorname{span}(C-c)
$
be its associated direction space. Since \(\mathcal K_cC\subset V\) and
\(C-c\subset \mathcal K_cC\), we get
$
V=\operatorname{span}(C-c)\subset \mathcal K_cC\subset V.
$
Therefore, $
\mathcal K_cC=V
$
for every \(c\in C\).

We claim that every \(c\in C\) is in the relative interior of \(C\) inside
\(A\). Indeed, if some \(c\in C\) were not in \(\mathrm{int}_A(C)\), the
supporting hyperplane theorem would give a nonzero \(u\in V\) such that $\langle u,z-c\rangle_2\le 0\  
  \text{for every }z\in C
$. Passing to the tangent cone gives
$
\langle u,h\rangle_2\le 0\,\,
 \text{for every }h\in \mathcal K_cC.
$
But \(\mathcal K_cC=V\), and \(u\in V\), so choosing \(h=u\) gives $
\|u\|_2^2\le0,
$
a contradiction. Hence \(C=\mathrm{int}_A(C)\).
Thus \(C\) is both relatively open and relatively closed in its affine hull
\(A\). Since \(A\) is connected and \(C\neq\emptyset\), it follows that $
C=A.
$
Therefore \(C\) is affine. Finally, since \(T^{-1}\) is an invertible linear map, $
\mathscr W=T^{-1}(C)
$
is affine as well.
This proves the converse direction, and hence the proposition.
\end{proof}

The result in \Cref{prop:differentiability_optimalconstrained} is particularly
relevant for structured versions of \Cref{eq:bilevel_varpro_formulation} in
which the eliminated variable is constrained to a structured matrix set
$\mathscr W$ and first-order methods are used. The proposition shows that global differentiability of the lower-level solution map is restrictive: for arbitrary right-hand sides, it essentially forces the constraint set to be affine.
In \Cref{subsec:blind_deconvolution} we present a numerical example in which the constraint set $\mathscr W$ is given by an affine subspace, representing convolutional linear transforms.

\section{The algorithm}

In this section, we present the proposed Constrained Regularized Variable Projection method. The method
is applied to the reduced problem
\begin{equation}\label{probFW}
    \min_{\theta\in \mathcal  C_A\times \mathbb{R}^{p_B}} \widehat f_U(\theta),
    \qquad
    \theta=(\theta_A,\theta_B),
\end{equation}
where \(\mathcal C_A\) is compact and convex, while the second block is unconstrained. The reduced
objective is defined as
\[
    \widehat f_U(\theta)
    :=
    f_U(\widehat w(\theta),\theta),
    \qquad
    \widehat w(\theta)\in \arg\min_w f_L(w,\theta).
\]
Therefore, the lower-level variable \(w\) is eliminated and the algorithm only updates the
outer variable \(\theta\). The gradient used in the algorithm is the hypergradient
\(\nabla \widehat f_U(\theta)\), which includes the dependence of the lower-level solution
\(\widehat w(\theta)\) on \(\theta\). The detailed scheme is reported in Algorithm \ref{alg:constrained_varpro}.

\label{sec:alg}
\begin{algorithm}[t]
\caption{Constrained Regularized Variable Projection (CR-VarPro)}
\label{alg:constrained_varpro}
\begin{algorithmic}
\STATE{Input: $\theta^{(0)} \in \mathcal C$, $T_{max}$}
\STATE{Set $t = 0$}
\WHILE{$t\leq T_{max}$}
\STATE{Assemble hypergradient $\nabla\widehat f_U(\theta^{(t)})$}
\STATE{Set $\bar \theta^{(t)} = \text{BLMO}(\theta^{(t)}, \nabla\widehat f_U(\theta^{(t)}))$}
\IF{$\theta^{(t)}$ stationary}
\STATE{\textbf{STOP}}
\ENDIF
\STATE{Set $\theta^{(t+1)} = \theta^{(t)} + \alpha_t (\bar\theta^{(t)} - \theta^{(t)})$, with $\alpha_t \in (0, 1]$ stepsize chosen via a line search}
\STATE{Set $t = t + 1$}
\ENDWHILE
\RETURN $\theta^{(t)}$
\end{algorithmic}
\end{algorithm}

Starting from an initial feasible point \(\theta^{(0)}\in \mathcal C := \mathcal C_A\times\mathbb{R}^{p_B}\),
Algorithm~\ref{alg:constrained_varpro} sets \(t=0\) and repeats the following operations until either stationarity is
reached or the maximum number of iterations is exceeded. At iteration \(t\), the first step consists in assembling the hypergradient $\nabla \widehat f_U(\theta^{(t)})$. 

Once the hypergradient has been computed, Algorithm~\ref{alg:constrained_varpro} calls the Block Linear Minimization Oracle described in Algorithm~\ref{alg:BLMO} and sets
\[
    \bar\theta^{(t)}
    =
    \operatorname{BLMO}
    \left(
        \theta^{(t)},
        \nabla \widehat f_U(\theta^{(t)})
    \right).
\]
The algorithm then checks whether \(\theta^{(t)}\) is stationary. Equivalently, one may use the
Frank-Wolfe gap
\[
    g_t
    :=
    -
    \left\langle
        \nabla \widehat f_U(\theta^{(t)}),
        d^{(t)}
    \right\rangle
\]
as a stationarity measure, with  $d^{(t)}:=\bar\theta^{(t)}-\theta^{(t)}.$  If $g_t$ is lower than a given threshold, the algorithm stops. Otherwise, a stepsize
\(\alpha_t\in(0,1]\) is chosen by line search and the new iterate is computed as
$\theta^{(t+1)}
    =
    \theta^{(t)}
    +
    \alpha_t
    \left(
        \bar\theta^{(t)}-\theta^{(t)}
    \right).$

\begin{algorithm}[t]
\caption{Block Linear Minimization Oracle (BLMO)}
\label{alg:BLMO}
\begin{algorithmic}
\STATE{Input: $\tilde\theta=(\tilde\theta_A,\tilde\theta_B) \in \mathcal C_A \times \mathcal \R^{p_B}$, $\nabla \widehat f_U(\tilde\theta)$}
\STATE{Get $\bar \theta_A := \underset{{\theta_A \in \mathcal C_A}}{\arg\min} \langle \theta_A ,\nabla_{\theta_A} \widehat f_U(\tilde\theta) \rangle$}
\STATE{Get $\bar \theta_B := \underset{{\|\theta_B -\tilde \theta_B\|\leq 1}}{\arg\min} \langle \theta_B ,\nabla_{\theta_B} \widehat f_U(\tilde\theta) \rangle$} 
\RETURN $(\bar \theta_A,\bar \theta_B)$
\end{algorithmic}
\end{algorithm}

We now clarify the role of the BLMO reported in Algorithm \ref{alg:BLMO}. At iteration \(t\), Algorithm~\ref{alg:constrained_varpro}
computes the point  $\bar\theta^{(t)}$
and defines the search direction
\[
    d^{(t)}
    :=
    \bar\theta^{(t)}-\theta^{(t)}
    =
    \left(
        d_A^{(t)},d_B^{(t)}
    \right),
\]
where
\[
    d_A^{(t)}
    :=
    \bar\theta_A^{(t)}-\theta_A^{(t)},
    \qquad
    d_B^{(t)}
    :=
    \bar\theta_B^{(t)}-\theta_B^{(t)}.
\]

The first block of the BLMO coincides with the usual Frank-Wolfe LMO
over the compact set \(\mathcal C_A\). The second block is different: instead of minimizing the linear
model over all of \(\mathbb{R}^{p_B}\), it minimizes over the unit ball centered at the current
iterate. Indeed, a standard LMO over the full feasible set would require solving
\[
    \min_{\theta_A\in \mathcal C_A,\ \theta_B\in\mathbb{R}^{p_B}}
    \left\langle
        \theta_A,
        \nabla_{\theta_A}\widehat f_U(\theta^{(t)})
    \right\rangle
    +
    \left\langle
        \theta_B,
        \nabla_{\theta_B}\widehat f_U(\theta^{(t)})
    \right\rangle .
\]
The first term is well-defined because \(\mathcal C_A\) is compact, but the second term is unbounded
from below whenever $\nabla_{\theta_B}\widehat f_U(\theta^{(t)})\neq 0$. 
Thus, the standard Frank-Wolfe LMO is not well-defined on
\(\mathcal C_A\times \mathbb{R}^{p_B}\). The BLMO avoids this issue by keeping the Frank-Wolfe
oracle on the compact set $\mathcal C_A$ and replacing the unbounded linear minimization
problem by a local one on the unconstrained variables. In particular, if
\(\nabla_{\theta_B}\widehat f_U(\theta^{(t)})\neq 0\), then
\[
    \bar\theta_B^{(t)}
    =
    \theta_B^{(t)}
    -
    \frac{
        \nabla_{\theta_B}\widehat f_U(\theta^{(t)})
    }{
        \left\|
        \nabla_{\theta_B}\widehat f_U(\theta^{(t)})
        \right\|
    },
\]
while if \(\nabla_{\theta_B}\widehat f_U(\theta^{(t)})=0\), one can choose
\(\bar\theta_B^{(t)}=\theta_B^{(t)}\). Consequently, the BLMO combines a Frank-Wolfe step on
\(\mathcal C_A\) with a normalized gradient descent step on \(\mathbb{R}^{p_B}\). The latter can be interpreted as
a norm-constrained LMO step, since it is equivalent to
\[
    d_B^{(t)}
    \in
    \arg\min_{\|d_B\|\le 1}
    \left\langle
        d_B,
        \nabla_{\theta_B}\widehat f_U(\theta^{(t)})
    \right\rangle .
\]
This is the update principle used in SCION-type methods, where descent directions are generated through LMOs over norm balls rather than by using the
raw gradient directly~\cite{pethick2025training}.

\section{Theoretical analysis}
\label{sec:main}
We now analyze Algorithm~\ref{alg:constrained_varpro}. Throughout this section, the objective
is the reduced function \(\widehat f_U\). We assume that \(\nabla \widehat f_U\) is Lipschitz continuous
with constant \(L>0\) on the level set generated by the algorithm, and that \(\widehat f_U\) is
bounded from below on \(\mathcal C_A\times \mathbb{R}^{p_B}\). We denote
\[
    \widehat f_U^*
    :=
    \inf_{\theta\in \mathcal C_A\times \mathbb{R}^{p_B}}
    \widehat f_U(\theta).
\]
Let $    \Delta_A:=\max_{\theta_A,\tilde\theta_A\in \mathcal C_A}\|\theta_A-\tilde\theta_A\|$
be the diameter of the compact block. Since the BLMO satisfies $\|d_B^{(t)}\|\le 1$, 
we define $    \bar\Delta := \sqrt{\Delta_A^2+1}$.
Then, for every iteration \(t\),
\[
    \|d^{(t)}\|^2
    =
    \|d_A^{(t)}\|^2+\|d_B^{(t)}\|^2
    \le
    \Delta_A^2+1
    =
    \bar\Delta^2.
\] 
Note that $g_t$ represents a valid stationarity measure for the product set
\(\mathcal C_A\times \mathbb{R}^{p_B}\). Indeed, by the construction of the BLMO,
\[
    -
    \left\langle
        \nabla_{\theta_A}\widehat f_U(\theta^{(t)}),
        d_A^{(t)}
    \right\rangle
    \ge 0,
    \qquad
    -
    \left\langle
        \nabla_{\theta_B}\widehat f_U(\theta^{(t)}),
        d_B^{(t)}
    \right\rangle
    \ge 0.
\]
Moreover, \(g_t=0\) if and only if $\left\langle
        \nabla_{\theta_A}\widehat f_U(\theta^{(t)}),
        \theta_A-\theta_A^{(t)}
    \right\rangle
    \ge 0$, $\forall \theta_A\in \mathcal C_A$, and $\nabla_{\theta_B}\widehat f_U(\theta^{(t)})=0$.
Thus, \(g_t=0\) is equivalent to the first-order stationarity condition for the constrained
block \(\mathcal C_A\) together with the unconstrained stationarity condition for the block
\(\mathbb{R}^{p_B}\).

Finally, we assume that the stepsize rule satisfies the conditions:
\begin{align}\label{eq:alpha_cond}
    &\alpha_t \geq
    \bar\alpha_t
    :=
    \min\left\{
        1,
        \frac{g_t}{L\bar\Delta^2}
    \right\} ,\\
    \label{eq:sufficient_decrease}
    &\widehat f_U(\theta^{(t)})
    -
    \widehat f_U(\theta^{(t+1)})
    \ge
    \rho \bar\alpha_t g_t
\end{align}
for some fixed \(\rho>0\).

\begin{Th}\label{th:convergence}
Let \(\{\theta^{(t)}\}\) be the sequence generated by
Algorithm~\ref{alg:constrained_varpro}. Assume that \(\nabla \widehat f_U\) is Lipschitz
continuous with constant \(L>0\), that \(\widehat f_U\) is bounded below by \(\widehat f_U^*\), and
that the stepsize satisfies~\eqref{eq:alpha_cond} and~\eqref{eq:sufficient_decrease}. Define $g_T^*
    :=
    \min_{0\le t\le T-1} g_t$. Then, for every \(T\in\mathbb{N}\),
\begin{equation}\label{eq:rate}
    g_T^*
    \le
    \max\left\{
        \sqrt{
        \frac{
            L\bar\Delta^2
            \left(
                \widehat f_U(\theta^{(0)})-\widehat f_U^*
            \right)
        }{
            \rho T
        }},
        \frac{
            2 \left (\widehat f_U(\theta^{(0)})-\widehat f_U^*
        \right)}{
            T
        }
    \right\}.
\end{equation}
\end{Th}

\begin{proof}
In order to prove the result, we distinguish two different cases.

\textbf{Case 1.} \(\bar{\alpha}_t<1\).

Then, by definition of \(\bar{\alpha}_t\), we have
$\bar{\alpha}_t
    =
    \frac{g_t}{L\bar\Delta^2}.
$
Using the sufficient decrease condition~\eqref{eq:sufficient_decrease}, we get
\begin{equation}\label{eq:case1}
    \widehat f_U(\theta^{(t)})
    -
    \widehat f_U(\theta^{(t+1)})
    =
    \widehat f_U(\theta^{(t)})
    -
    \widehat f_U(\theta^{(t)}+\alpha_t d^{(t)})
    \ge
    \frac{\rho g_t^2}{L\bar\Delta^2}.
\end{equation}

\textbf{Case 2.} \(\bar{\alpha}_t=1\).

Since \(\alpha_t\in(0,1]\) and \(\alpha_t\ge \bar{\alpha}_t\), the condition
\(\bar{\alpha}_t=1\) implies \(\alpha_t=1\). By the standard descent lemma~\cite[Proposition 6.1.2]{bertsekas2015convex} applied to
\(\widehat f_U\) with center \(\theta^{(t)}\) and direction \(d^{(t)}\), we have
\[
\begin{aligned}
    \widehat f_U(\theta^{(t+1)})
    &=
    \widehat f_U(\theta^{(t)}+d^{(t)})                                      \\
    &\le
    \widehat f_U(\theta^{(t)})
    +
    \left\langle
        \nabla \widehat f_U(\theta^{(t)}),
        d^{(t)}
    \right\rangle
    +
    \frac{L}{2}\|d^{(t)}\|^2                                            \\
    &=
    \widehat f_U(\theta^{(t)})
    +
    \left\langle
        \nabla_{\theta_A}\widehat f_U(\theta^{(t)}),
        d_A^{(t)}
    \right\rangle
    +
    \left\langle
        \nabla_{\theta_B}\widehat f_U(\theta^{(t)}),
        d_B^{(t)}
    \right\rangle
    +
    \frac{L}{2}
    \left(
        \|d_A^{(t)}\|^2+\|d_B^{(t)}\|^2
    \right)                                                            \\
    &\le
    \widehat f_U(\theta^{(t)})
    -
    g_t
    +
    \frac{L}{2}\bar\Delta^2 .
\end{aligned}
\]
In the last inequality, we used the definition
$ g_t
    :=
    -
    \left\langle
        \nabla \widehat f_U(\theta^{(t)}),
        d^{(t)}
    \right\rangle
$
and the bound
\[
    \|d^{(t)}\|^2
    =
    \|d_A^{(t)}\|^2+\|d_B^{(t)}\|^2
    \le
    \Delta_A^2+1
    =
    \bar\Delta^2.
\]
Moreover, since
$\bar{\alpha}_t
    =
    \min\left\{
        1,
        \frac{g_t}{L\bar\Delta^2}
    \right\}
    =
    1,
$
we have $\frac{g_t}{L\bar\Delta^2}\ge 1$, and hence $g_t\ge L\bar\Delta^2$.

Therefore,
\begin{equation}\label{eq:case2}
    \widehat f_U(\theta^{(t)})
    -
    \widehat f_U(\theta^{(t+1)})
    \ge
    g_t-\frac{L}{2}\bar\Delta^2
    \ge
    \frac{g_t}{2}.
\end{equation}

Now, based on the two cases above, we partition the iterations
\(\{0,1,\ldots,T-1\}\) into
\[
    N_1
    :=
    \{t<T:\bar{\alpha}_t<1\},
    \qquad
    N_2
    :=
    \{t<T:\bar{\alpha}_t=1\}.
\]
Using~\eqref{eq:case1} and~\eqref{eq:case2}, we obtain
\[
\begin{aligned}
    \widehat f_U(\theta^{(0)})-\widehat f_U^*
    &\ge
    \sum_{t=0}^{T-1}
    \left(
        \widehat f_U(\theta^{(t)})
        -
        \widehat f_U(\theta^{(t+1)})
    \right)                                                        \\
    &=
    \sum_{t\in N_1}
    \left(
        \widehat f_U(\theta^{(t)})
        -
        \widehat f_U(\theta^{(t+1)})
    \right)
    +
    \sum_{t\in N_2}
    \left(
        \widehat f_U(\theta^{(t)})
        -
        \widehat f_U(\theta^{(t+1)})
    \right)                                                        \\
    &\ge
    \sum_{t\in N_1}
    \frac{\rho g_t^2}{L\bar\Delta^2}
    +
    \sum_{t\in N_2}
    \frac{g_t}{2}                                                   \\
    &\ge
    |N_1|
    \min_{t\in N_1}
    \frac{\rho g_t^2}{L\bar\Delta^2}
    +
    |N_2|
    \min_{t\in N_2}
    \frac{g_t}{2}                                                   \\
    &\ge
    \left(|N_1|+|N_2|\right)
    \min\left\{
        \frac{\rho (g_T^*)^2}{L\bar\Delta^2},
        \frac{g_T^*}{2}
    \right\}                                                        \\
    &=
    T
    \min\left\{
        \frac{\rho (g_T^*)^2}{L\bar\Delta^2},
        \frac{g_T^*}{2}
    \right\},
\end{aligned}
\]
where in the last inequality we used the definition of $g_T^*$.
Hence,
\[
    T
    \min\left\{
        \frac{\rho (g_T^*)^2}{L\bar\Delta^2},
        \frac{g_T^*}{2}
    \right\}
    \le
    \widehat f_U(\theta^{(0)})-\widehat f_U^*.
\]

To finish, if $    T
    \min\left\{
        \frac{\rho (g_T^*)^2}{L\bar\Delta^2},
        \frac{g_T^*}{2}
    \right\}
    =
    T\frac{g_T^*}{2}$,
then
\begin{equation}\label{eq:gstar_linear}
    g_T^*
    \le
    \frac{
        2\left(
            \widehat f_U(\theta^{(0)})-\widehat f_U^*
        \right)
    }{T}.
\end{equation}
Otherwise,
\begin{equation}\label{eq:gstar_sqrt}
    g_T^*
    \le
    \sqrt{
        \frac{
            L\bar\Delta^2
            \left(
                \widehat f_U(\theta^{(0)})-\widehat f_U^*
            \right)
        }{
            \rho T
        }
    }.
\end{equation}
The claim follows by taking the maximum in the system formed by
\eqref{eq:gstar_linear} and~\eqref{eq:gstar_sqrt}.
\end{proof}

The optimality condition follows from the definition of \(g_t\). Indeed, by construction of the
BLMO, both quantities
\[
    -
    \left\langle
        \nabla_{\theta_A}\widehat f_U(\theta^{(t)}),
        d_A^{(t)}
    \right\rangle,
    \qquad
    -
    \left\langle
        \nabla_{\theta_B}\widehat f_U(\theta^{(t)}),
        d_B^{(t)}
    \right\rangle
\]
are nonnegative. Therefore, if \(g_t=0\), then both of them must be equal to zero.

For the constrained block, we have
\[
    -
    \left\langle
        \nabla_{\theta_A}\widehat f_U(\theta^{(t)}),
        d_A^{(t)}
    \right\rangle
    =
    0.
\]
Since \(d_A^{(t)}=\bar\theta_A^{(t)}-\theta_A^{(t)}\) is generated by the Frank-Wolfe
linear oracle on \(\mathcal C_A\), it satisfies
\[
    \left\langle
        \nabla_{\theta_A}\widehat f_U(\theta^{(t)}),
        \bar\theta_A^{(t)}
    \right\rangle
    \le
    \left\langle
        \nabla_{\theta_A}\widehat f_U(\theta^{(t)}),
        \theta_A
    \right\rangle
    \qquad
    \forall \theta_A\in \mathcal C_A.
\]
Equivalently,
\[
    \left\langle
        \nabla_{\theta_A}\widehat f_U(\theta^{(t)}),
        d_A^{(t)}
    \right\rangle
    \le
    \left\langle
        \nabla_{\theta_A}\widehat f_U(\theta^{(t)}),
        \theta_A-\theta_A^{(t)}
    \right\rangle
    \qquad
    \forall \theta_A\in \mathcal C_A.
\]
Since the left-hand side is zero, we obtain
\[
    \left\langle
        \nabla_{\theta_A}\widehat f_U(\theta^{(t)}),
        \theta_A-\theta_A^{(t)}
    \right\rangle
    \ge 0
    \qquad
    \forall \theta_A\in \mathcal C_A.
\]
Thus, the first-order optimality condition holds with respect to the constrained block
\(\theta_A\).

For the unconstrained block, we have
\[
    -
    \left\langle
        \nabla_{\theta_B}\widehat f_U(\theta^{(t)}),
        d_B^{(t)}
    \right\rangle
    =
    0.
\]
The BLMO defines \(d_B^{(t)}\) as the solution of
\[
    d_B^{(t)}
    \in
    \arg\min_{\|d_B\|\le 1}
    \left\langle
        d_B,
        \nabla_{\theta_B}\widehat f_U(\theta^{(t)})
    \right\rangle .
\]
Hence, if \(\nabla_{\theta_B}\widehat f_U(\theta^{(t)})\neq 0\), then
$d_B^{(t)}
    =
    -
    \frac{
        \nabla_{\theta_B}\widehat f_U(\theta^{(t)})
    }{
        \left\|
            \nabla_{\theta_B}\widehat f_U(\theta^{(t)})
        \right\|
    },
$
and therefore
\[
    -
    \left\langle
        \nabla_{\theta_B}\widehat f_U(\theta^{(t)}),
        d_B^{(t)}
    \right\rangle
    =
    \left\|
        \nabla_{\theta_B}\widehat f_U(\theta^{(t)})
    \right\|
    >
    0,
\]
which contradicts the equality above. Consequently, $\nabla_{\theta_B}\widehat f_U(\theta^{(t)})=0.$
Therefore, \(g_t=0\) is equivalent to the stationarity conditions
\[
    \left\langle
        \nabla_{\theta_A}\widehat f_U(\theta^{(t)}),
        \theta_A-\theta_A^{(t)}
    \right\rangle
    \ge 0
    \qquad
    \forall \theta_A\in \mathcal C_A,
\]
and
\[
    \nabla_{\theta_B}\widehat f_U(\theta^{(t)})=0.
\]

In the following lemma, we show that the conditions~\eqref{eq:alpha_cond} and~\eqref{eq:sufficient_decrease} 
can be satisfied by
standard stepsize rules. In particular, this is true when \(\alpha_t=\bar\alpha_t\), and also
when \(\alpha_t\) is determined by an Armijo line search (see \cite{bomze2020active,bomze2021frank} for further details). The Armijo rule defines
\begin{equation}\label{eta}
    \alpha_t=\delta^j,
\end{equation}
where \(j\) is the smallest nonnegative integer such that
\begin{equation}\label{Armijo}
    \widehat f_U(\theta^{(t)})
    -
    \widehat f_U(\theta^{(t)}+\alpha_t d^{(t)})
    \ge
    \gamma \alpha_t g_t.
\end{equation}
Here, \(\gamma\in(0,\tfrac12)\) and \(\delta\in(0,1)\) are fixed constants.

\begin{Lemma}\label{stepsize}
The bound condition on the stepsize
\begin{equation}\label{eq:1lb}
    \alpha_t \geq
    \bar\alpha_t
    :=
    \min\left\{
        1,
        \frac{g_t}{L\bar\Delta^2}
    \right\} ,
\end{equation}

and the sufficient decrease condition
\begin{equation}\label{eq:2ar}
    \widehat f_U(\theta^{(t)})
    -
    \widehat f_U(\theta^{(t+1)})
    \ge
    \rho \bar\alpha_t g_t
\end{equation}
hold under the following stepsize rules:
\begin{itemize}
\item If \(\alpha_t=\bar\alpha_t\), then the condition holds
with \(\rho=\frac12\). 
\item If \(\alpha_t\) is determined by the Armijo line search rule
\eqref{Armijo}, then the condition holds with
\[
    \rho=\gamma\min\{1,2\delta(1-\gamma)\}.
\]
\end{itemize}
\end{Lemma}

\begin{proof}
By the standard descent lemma~\cite[Proposition 6.1.2]{bertsekas2015convex}, for every
\(\alpha\in[0,1]\), we have
\begin{equation}\label{e11:stdnew}
\begin{aligned}
    \widehat f_U(\theta^{(t)})
    -
    \widehat f_U(\theta^{(t)}+\alpha d^{(t)})
    &\ge
    -
    \alpha
    \left\langle
        \nabla \widehat f_U(\theta^{(t)}),
        d^{(t)}
    \right\rangle
    -
    \alpha^2
    \frac{L}{2}
    \|d^{(t)}\|^2                                      \\
    &=
    -
    \alpha
    \left\langle
        \nabla_{\theta_A}\widehat f_U(\theta^{(t)}),
        d_A^{(t)}
    \right\rangle
    -
    \alpha
    \left\langle
        \nabla_{\theta_B}\widehat f_U(\theta^{(t)}),
        d_B^{(t)}
    \right\rangle                                      \\
    &\quad
    -
    \alpha^2
    \frac{L}{2}
    \left(
        \|d_A^{(t)}\|^2+\|d_B^{(t)}\|^2
    \right)                                            \\
    &=
    \alpha g_t
    -
    \alpha^2
    \frac{L}{2}
    \left(
        \|d_A^{(t)}\|^2+\|d_B^{(t)}\|^2
    \right).
\end{aligned}
\end{equation}

First, assume that \(\alpha_t=\bar\alpha_t\). Then the stepsize lower bound \eqref{eq:1lb} is trivially
satisfied. Moreover, from~\eqref{e11:stdnew}, it is immediate that
\begin{equation}\label{eq11:linnew}
    \alpha g_t
    -
    \alpha^2
    \frac{L}{2}
    \left(
        \|d_A^{(t)}\|^2+\|d_B^{(t)}\|^2
    \right)
    \ge
    \alpha \frac{g_t}{2}
\end{equation}
for every
\[
    0\le \alpha
    \le
    \frac{
        g_t
    }{
        L\left(
            \|d_A^{(t)}\|^2+\|d_B^{(t)}\|^2
        \right)
    }.
\]
We can apply~\eqref{eq11:linnew} to \(\bar\alpha_t\), since
\[
    0
    \le
    \bar\alpha_t
    \le
    \frac{g_t}{L\bar\Delta^2}
    \le
    \frac{
        g_t
    }{
        L\left(
            \|d_A^{(t)}\|^2+\|d_B^{(t)}\|^2
        \right)
    },
\]
where the last inequality follows from $\|d_A^{(t)}\|^2+\|d_B^{(t)}\|^2
    \le
    \bar\Delta^2$. Therefore,
\[
\begin{aligned}
    \widehat f_U(\theta^{(t)})
    -
    \widehat f_U(\theta^{(t+1)})
    =
    \widehat f_U(\theta^{(t)})
    -
    \widehat f_U(\theta^{(t)}+\bar\alpha_t d^{(t)})      \ge
    \bar\alpha_t\frac{g_t}{2}.
\end{aligned}
\]
Thus, the sufficient decrease \eqref{eq:2ar} condition holds with \(\rho=\frac12\).

Now assume that \(\alpha_t\) is determined by the Armijo line search rule. From
\eqref{e11:stdnew}, the Armijo condition $\widehat f_U(\theta^{(t)})
    -
    \widehat f_U(\theta^{(t)}+\alpha d^{(t)})
    \ge
    \gamma \alpha g_t$, 
is satisfied whenever
\[
    0
    \le
    \alpha
    \le
    2(1-\gamma)
    \frac{
        g_t
    }{
        L\left(
            \|d_A^{(t)}\|^2+\|d_B^{(t)}\|^2
        \right)
    }.
\]
By the standard backtracking argument, the accepted stepsize satisfies
\[
    \alpha_t
    \ge
    \min\left\{
        1,
        2\delta(1-\gamma)
        \frac{
            g_t
        }{
            L\left(
                \|d_A^{(t)}\|^2+\|d_B^{(t)}\|^2
            \right)
        }
    \right\}.
\]
Using again $    \|d_A^{(t)}\|^2+\|d_B^{(t)}\|^2
    \le
    \bar\Delta^2$, we get
\begin{equation}\label{1Ar}
\begin{aligned}
    \alpha_t \ge
    \min\left\{
        1,
        2\delta(1-\gamma)
        \frac{g_t}{L\bar\Delta^2}
    \right\} \ge
    \min\{1,2\delta(1-\gamma)\}\bar\alpha_t,
\end{aligned}
\end{equation}
we thus have $\alpha_t \geq
    \min\left\{
        1,c
        \frac{g_t}{L\bar\Delta^2}
    \right\}$,

for some $c>0$. We have two cases: if $c\geq1$ the lower bound \eqref{eq:1lb} is trivially satisfied. If $c<1$ we can still satisfy equation \eqref{eq:1lb} by considering $\tilde L = L/c$ instead of $L$ as Lipschitz constant. Finally, using the Armijo condition~\eqref{Armijo} and then~\eqref{1Ar}, we obtain
\[
\begin{aligned}
    \widehat f_U(\theta^{(t)})
    -
    \widehat f_U(\theta^{(t+1)})
    &=
    \widehat f_U(\theta^{(t)})
    -
    \widehat f_U(\theta^{(t)}+\alpha_t d^{(t)})          \ge
    \gamma \alpha_t g_t              \geq                \\
    &\ge
    \gamma
    \min\{1,2\delta(1-\gamma)\}
    \bar\alpha_t g_t.
\end{aligned}
\]
Hence the sufficient decrease condition \eqref{eq:2ar} holds with $ \rho = \gamma\min\{1,2\delta(1-\gamma)\}$.
\end{proof}

\section{Experimental results}
\label{sec:experiments}

\subsection{Example 1: Sparse autoencoder and latent dimension selection}

Autoencoding has a variety of different applications, from representation learning \cite{tschannen2018recentadvancesautoencoderbasedrepresentation}, dimensionality reduction \cite{hinton_autoencoder}, image denoising and generation \cite{bengio2013generalizeddenoisingautoencodersgenerative}, to neural network interpretability \cite{huben2024sparse,alain2017understanding}. The problem of autoencoding has a number of variants, but in its basic version, it consists of reconstructing the identity map on a dataset through the composition of learnable encoding and decoding functions as a solution of the optimization problem
$
\min_{\theta \in \mathcal C} \frac{1}{2} \| D \circ E(X)-Y \|_F^2,
$
where $E$ and $D$ are the encoder and decoder neural networks, respectively. In this experiment, we consider the asymmetric setting with $D_\phi(x) = Wx$, $W \in \R^{D \times d}$, and $E_{\theta}(x)$ is a feedforward neural network with a hyperbolic tangent activation function. In this setting, the autoencoding problem can be reformulated in the setting of \Cref{eq:bilevel_general_formulation} as
\begin{equation*}
    \begin{cases}
 \underset{e^\top s \leq \rho, s \geq 0,\phi \in \R^p}{\min} \| \widehat  W \diag(s) E_\phi(X_1) - X_1 \|_F^2   +\lambda \|s \|^2 \\
 s.t.\,\,\widehat  W \in \underset{W \in \R^{D \times d}}{\arg\min} \|W\diag(s) E_\phi(X_2) - X_2 \|_F^2,
\end{cases}
\end{equation*}
which fits exactly the formulation in \Cref{eq:bilevel_general_formulation} with $\theta = (s, \phi) \in \R^{d+p}$,$w = \mathrm{vec}(W)$, $M_U(\theta) = E_\phi(X_1)^\top \diag(s) \otimes I$,$M_L(\theta) = E_\phi(X_2)^\top \diag(s) \otimes I$, $y_U \equiv X_1, y_L \equiv X_2$ and $\mathcal C_A = B_{\|\cdot \|_1}(0,\rho] \cap \R^d_{\geq 0}, \mathcal C_B = \R^p$.

In \Cref{fig:autoencoder}, we report the results of the method presented in \Cref{alg:constrained_varpro} in the setting in which $E$ is a four-layer neural network with intermediate dimensions $[256,128,64,32]$, $\rho = 1, \lambda = 10^{-5}$, and the datasets $X_1,X_2$ are two splits of the MNIST training dataset. We compare the performance of Algorithm \ref{alg:constrained_varpro} with a standard projected gradient descent on the joint problem \ref{eq:joint_formulation} with a budget of $20$K optimization steps. As we can observe from the results in \Cref{fig:autoencoder}, despite the higher per-iteration computational cost, \algname{} is able to achieve a lower test reconstruction loss given a fixed time budget. We repeat the same experiment with the CIFAR10 dataset for $12$K optimization steps, with the only difference that $E$ is a four-layer neural network with doubled intermediate, i.e., $[512,256,128,64]$.

\begin{figure}[t]
    \centering
    \begin{tabular}{cc}
        \includegraphics[width=0.47\linewidth]{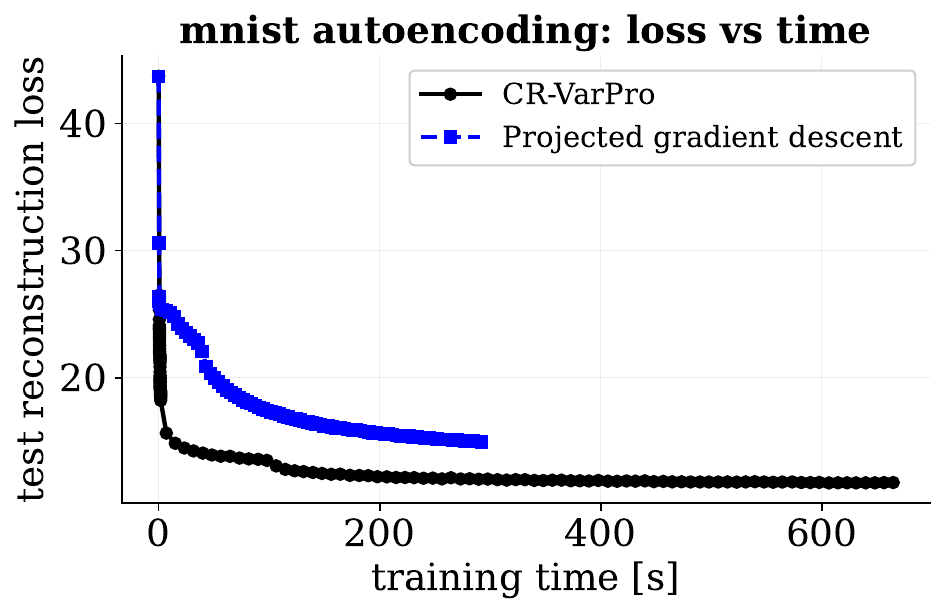} & 
        \includegraphics[width=0.47\linewidth]{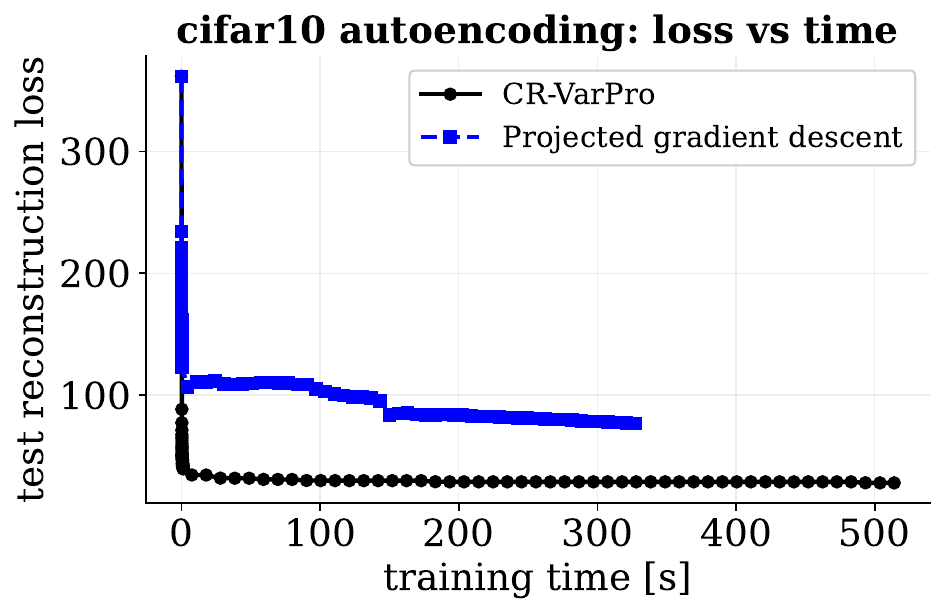}
        \\
        (a) MNIST \cite{lecun1998mnist}
        &
        (b) CIFAR10 \cite{krizhevsky2009learning}
    \end{tabular}
    \caption{Convergence of CR-VarPro against Projected gradient descent for the Autoencoding problem of MNIST and CIFAR10 on a fully connected neural network.}
    \label{fig:autoencoder}
\end{figure}

\subsection{Example 2: Dictionary learning}
We consider here the dictionary-learning problem \cite{Rubinstein_DL}, where the goal is to find a factorization of a given matrix $X \in \R^{D \times N}$ into a dictionary $D \in \R^{D \times d}$ and a norm-constrained representation matrix $R \in \R^{d \times N}$. This problem has many applications in data science and machine learning, including recommender systems \cite{dictionary_learning_reccommender}, data compression \cite{Rubinstein_DL}, and interpretability of large language models \cite{huben2024sparse}. We consider here the specialized setting in which we are interested in the reconstruction error in terms of the Frobenius norm
\begin{equation}\label{problem:Dictionary_learning_1}
\min_{D \in \R^{D \times d}, R \in \R^{d \times N}} \|DR-X \|^2 + \lambda \|\Omega R \|^2, \quad \text{s.t.}\,\, \|D \|_1 \leq \rho,
\end{equation}
where $\| \cdot \|_1$ indicates the entrywise $L^1$ norm. Notice that this problem is exactly a special case of the formulation \Cref{eq:varpro_general_formulation} with $\theta = \mathrm{vec}(D),M(\theta) = I \otimes D,w = \mathrm{vec}(R),y = \mathrm{vec}(X)$.
 In particular, the minimizers of \Cref{problem:Dictionary_learning_1} are minimizers of the bilevel-reformulated version
 \begin{equation}\label{problem:Dictionary_learning_2}
     \begin{cases}
         \underset{D \in \R^{D \times d}}{\min} \|D \widehat R-X \|^2+ \lambda \|\Omega \widehat R\|^2, \\
         \text{s.t.}\quad \widehat R \in \underset{R \in \R^{d \times N}}{\arg\min} \|DR-X \|^2 + \lambda \|\Omega R \|^2 , \quad \|D \|_1 \leq \rho .
     \end{cases}
 \end{equation}

 The lower-level optimization problem can be solved in closed form, leading to $\widehat R=(D^\top D + \lambda \Omega^\top \Omega)^{-1} D^\top X$. In particular, \Cref{problem:Dictionary_learning_2} can be reformulated as the following optimization problem
 \begin{equation}\label{eq:bilevel_dictionary_learning}
 \begin{cases}
      \underset{D \in \R^{D \times d}}{\min} \widehat f(D):=\|D(D^\top D + \lambda \Omega^\top \Omega)^{-1} D^\top X-X \|^2 + \lambda \|\Omega (D^\top D + \lambda \Omega^\top \Omega)^{-1} D^\top X \|^2, \\
      \text{s.t.}\,\, \quad \|D \|_1 \leq \rho .
 \end{cases}
 \end{equation}
 For this numerical experiment, we set $N =  3000, d = 128, D = 784$, $\Omega = I, \rho = 1, \lambda = 10^{-5}$ and $X \in \R^{D \times N}$ is a randomly sampled subset of the MNIST dataset \cite{lecun1998mnist}.
 In \Cref{fig:dict_learning_numerical} we present the numerical results of \Cref{eq:bilevel_dictionary_learning}, comparing \Cref{alg:constrained_varpro} with Projected gradient descent on the joint problem \Cref{problem:Dictionary_learning_1} for a total of $10$K optimization steps (with objective value reported every $25$ steps). 

\begin{figure}[t]
    \centering
    \begin{tabular}{c}
        \includegraphics[width=0.6\linewidth]{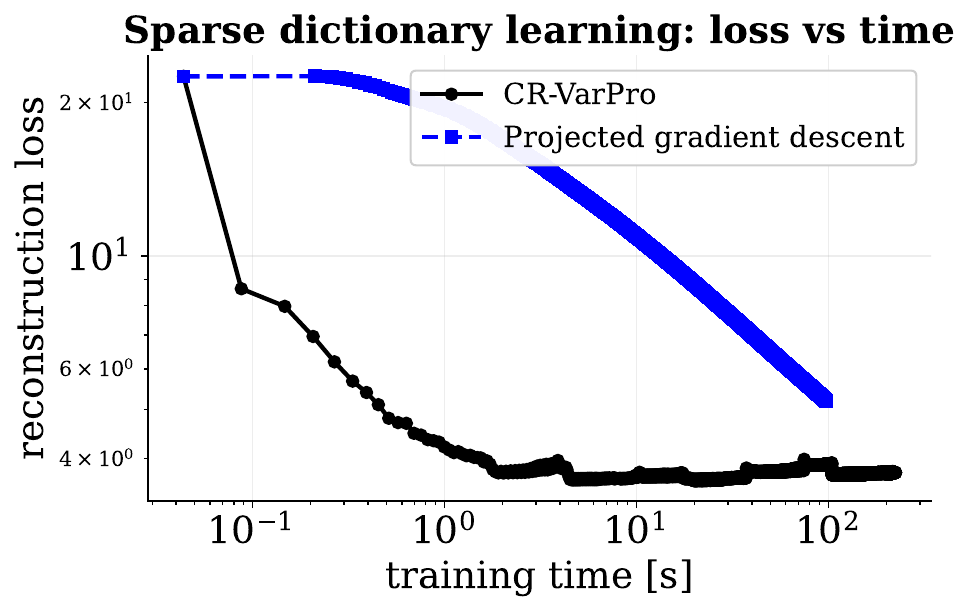}
    \end{tabular}
    \caption{Convergence of CR-VarPro against Projected gradient descent for the Dictionary Learning problem. The number of iterations is the same for the two methods, a budget of $10000$ optimization steps.}
    \label{fig:dict_learning_numerical}
\end{figure}
 \subsection{Example 3: Blind deconvolution}\label{subsec:blind_deconvolution}
Blind deconvolution is a common problem in image processing, which requires reconstructing a corrupted signal without prior knowledge on the smoothing process. More precisely, assume we are given a batch of $N$ measured signals $Y \in \R^{D \times N}$, and we know that they have been corrupted by a smoothing process and additional noise $y_i = w  *x_i + \varepsilon_i,$
and we assume to have no knowledge on the kernel $w$. The problem of blind deconvolution is trying to recover the batch of original signals $X = [x_1,\dots,x_N]$ given the corrupted ones $Y = [y_1,\dots,y_N]$. We remark that, even without any additional noise (i.e., $\varepsilon_i=0$) the problem is still ill posed. In particular, assuming $\varepsilon_i=0$, the problem becomes to find $x_i,w$ such that $w * x_i = y_i$ for all $i=1,\dots,N$. However, note  that given a pair $(x_i^\star,w^\star)$ that solves the problem, then $(\alpha x_i^\star,\alpha^{-1}w^\star)$ is still a solution for all $\alpha \ne 0$. More than this scalar invariance (which already produces a continuum of solutions), there is a full diagonal invariance, as
$\mathcal F(w * x_i) = \mathcal F(w) \odot \mathcal F(x_i) = \mathcal F(y_i)$
and any couple $(\mathcal F(w) \odot d, \mathcal F(x_i) \odot d^{\odot -1})$ is still a solution. In order to make the recovery problem well-posed, some additional constraints are typically needed. Assuming that the noise $\varepsilon_i \underset{i.i.d}{\sim} p$, and assuming that we have a prior knowledge on $w \sim q$ and we know that the original signal is constrained $x_i \in \mathcal C$ (e.g., box constraint or norm-based), reconstruction can be posed as a constrained maximum a posteriori estimation
\[
\min_{w,x_i \in \mathcal C} \sum_{i=1}^Np(y_i-w* x_i) +  q(w).
\]
In the specific case in which we assume a Gaussian prior $q(w) \propto e^{-\lambda \| w\|_2^2}$ and we assume that the noise is Gaussian as well, we get the objective function
$
\min_{W \in \mathscr W,X \in \mathcal C} \|WX-Y \|_F^2 + \lambda \|W \|_F^2,
$
leading back to the problem formulation we presented in \Cref{sec:structured_varpro}. In particular, if we assume to use circular convolution and $w \in \R^{D}, X \in \R^{D \times N}$, the problem can be solved in closed form in $W$ in the Fourier domain, and it reduces to a non-linear constrained minimization problem in $X$,
\begin{equation}\label{eq:blind_deconvolution_formulation}
\min_{X \in \mathcal C}  \| w^*(X) * X- Y ||_F^2 + \lambda \|w^*(X) \|_F^2, \quad w^*(X):=\mathcal F^{-1}\left\{\frac{\mathrm{diag}(\mathcal F\{ Y\}\mathcal F\{ X\}^H)}{\mathrm{diag}(\mathcal F\{ X\} \mathcal F\{ X\}^H + \lambda I)}
\right\},\end{equation}
which can be solved using the Frank-Wolfe algorithm. In \Cref{fig:blind_deconvolution}, we present convergence results (wall-time against objective function) for \Cref{alg:constrained_varpro} and Projected gradient descent. To do this, we considered the specific instance of the problem presented in \Cref{eq:blind_deconvolution_formulation} in the case in which $\mathcal C_A:= B_{\| \cdot \|_F}(0,\rho), \mathcal C_B = \{ 0\}$.
In \Cref{fig:blind_deconvolution} we report the results of wall time against objective function value for the blind deconvolution problem, for a total of $5$K iterations $\rho = 1, N = 50, D = 128, \lambda = 10^{-3}$.

\begin{figure}[t]
    \centering
    \includegraphics[width=0.6\linewidth]{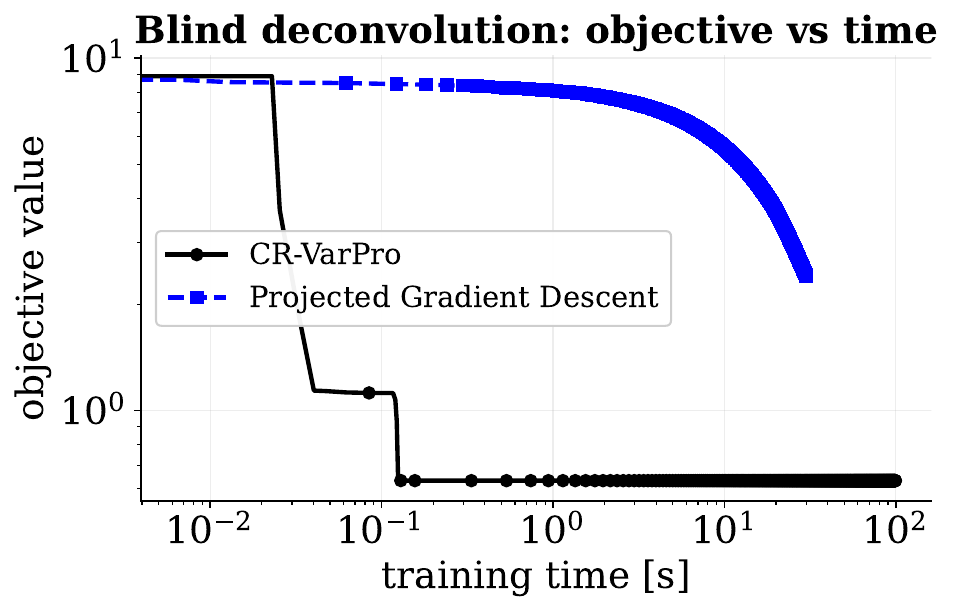}
    \caption{Convergence of CR-VarPro against Projected gradient descent for the blind deconvolution problem.}
    \label{fig:blind_deconvolution}
\end{figure}

\subsection{Example 4: Few-shot learning}

As a final numerical experiment, we test the performance of \Cref{alg:constrained_varpro} in few-shot learning on CIFAR10 \cite{krizhevsky2009learning} starting from a pretrained ResNet-18 \cite{resnet}, and compare it with other fine-tuning approaches. In particular, the few-shot learning problem fits the formulation \Cref{eq:bilevel_general_formulation} with $M_U(\theta) = \phi_\theta(X_1)^\top \otimes I,M_L(\theta) = \phi_\theta(X_2)^\top \otimes I, w = \mathrm{vec}(W)$ where $\phi$ is the backbone of the Resnet-18 architecture (all layers but the final linear classifier), and $X_1,X_2 \in \R^{32 \times 32 \times N}$ are two different splits of the CIFAR10 dataset. In this setting, $\theta$ is not composed of all trainable parameters of the backbone, but just the parameters of the fourth layer. For this particular experiment, we use a 10-way 1-shot setup, i.e., $N = 10$, and we use one example per class on each split. We compare the performance of \algname{} with fine-tuning only the last linear layer using an $L^2$ loss (Frozen ridge in \Cref{fig:few_shot_learning}), with cross-entropy loss (CE linear probe in \Cref{fig:few_shot_learning}), and by fine-tuning the last linear layer and the fourth layer of $\phi$.
In \Cref{fig:few_shot_learning}, we compare the final test accuracy of each method against the effective training time (not accounting for full accuracy calculation) in seconds, from which we can observe the effectiveness of the bilevel formulation in terms of data efficiency, training time, and overall performance.

\begin{figure}[t]
    \centering
    \includegraphics[width=0.65\linewidth]{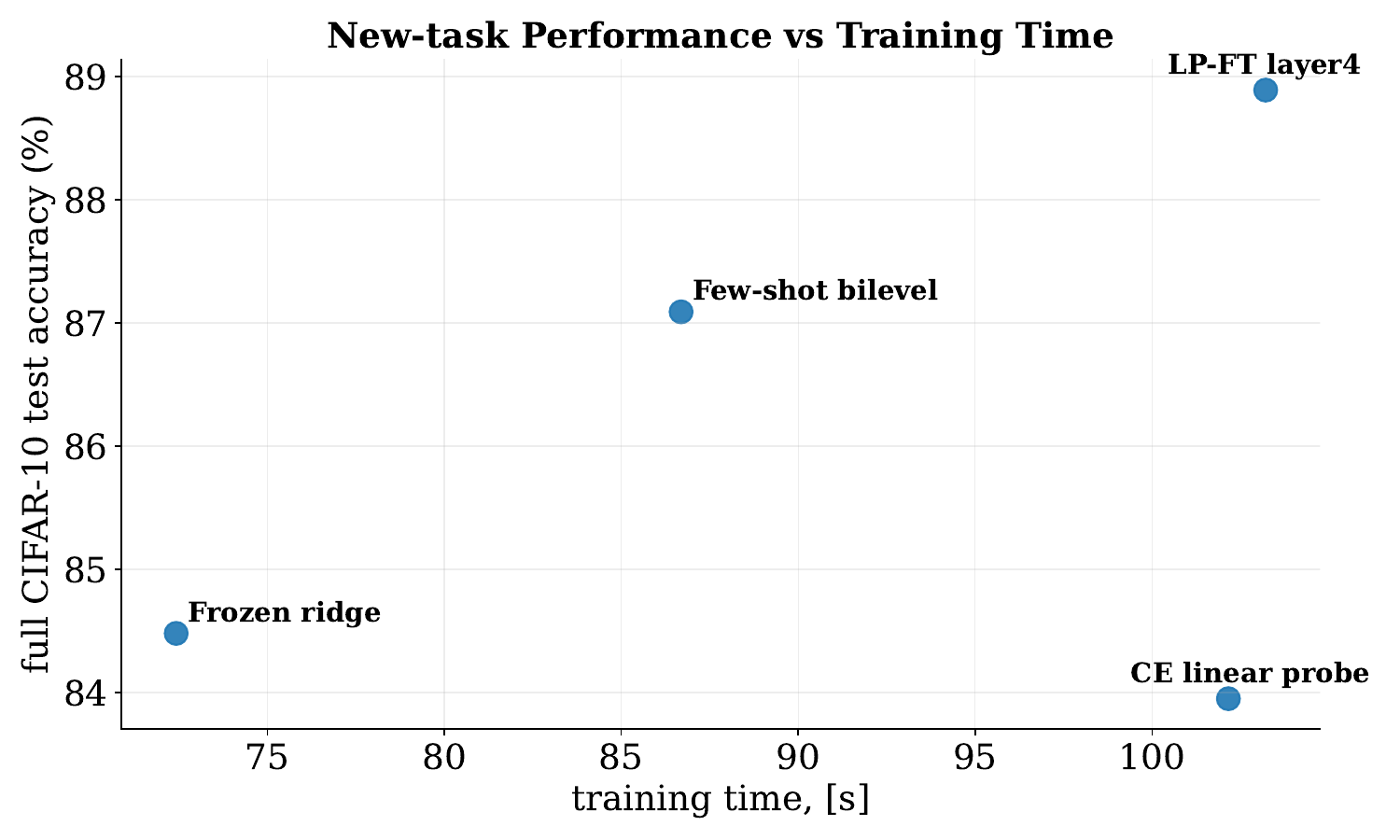}
    \caption{Fine tuning results ResNet-18 on CIFAR10. Frozen Ridge fits only the linear head as a solution of an $L^2$ regression problem, CE linear probe fits only the linear head through cross-entropy minimization, LP-FT layer4 fine-tunes the linear probe and the fourth intermediate layer. Few-shot bilevel performs 10-way 1-shot 1-query learning using the bilevel formulation in \Cref{eq:bilevel_general_formulation}.}
    \label{fig:few_shot_learning}
\end{figure}

\section{Conclusions}
\label{sec:conclusions}

We developed a constrained variable-projection framework for structured data-science models in which a least-squares block is eliminated exactly and the remaining variables are optimized over a convex feasible set. By interpreting variable projection as a collapsed bilevel problem, we derived reduced-gradient formulas that combine closed-form lower-level solves with automatic differentiation through vector-Jacobian products. This yields exact hypergradients without differentiating naively through normal equations, and naturally accommodates extensions in which the eliminated variable has additional affine structure.

We proposed a projection-free conditional-gradient method for the resulting reduced problem, combining a Frank-Wolfe oracle on the constrained block with a normalized descent step on unconstrained variables. Under standard smoothness and boundedness assumptions, we established convergence to first-order stationary points in terms of a Frank–Wolfe-type gap. The numerical experiments on sparse autoencoding, dictionary learning, blind deconvolution, and few-shot learning indicate that constrained variable projection can improve computational efficiency and data efficiency relative to natural joint-optimization baselines. Future work includes inexact lower-level solves, stochastic variants, and broader classes of structured lower-level constraints.

\paragraph{Acknowledgements}
The work of FT is partially funded by the PRIN-MUR project MOLE code 2022ZK5ME7 and by the PRIN-PNRR project FIN4GEO within the European Union's Next Generation EU framework, Mission 4, Component 2, CUP P2022BNB97. The work of E. Zangrando was funded by the MUR-PNRR project “Low-parametric machine learning”.

\bibliographystyle{siamplain}
\bibliography{references}

\end{document}